\documentclass[12pt,leqno]{amsart} 
\usepackage{amsmath,amstext,amsthm,amssymb,amsxtra}
\usepackage[colorlinks,citecolor=red,pagebackref,hypertexnames=false]{hyperref}
\usepackage[backrefs]{amsrefs}
\allowdisplaybreaks
 
\theoremstyle{plain} 
\newtheorem{lemma}[equation]{Lemma} 
 
\newtheorem{theorem}[equation]{Theorem}

\theoremstyle{definition}
\theoremstyle{example}
\newtheorem*{example}{Example}
\theoremstyle{remark}
\newtheorem*{remark}{Remark}

\newcommand{\assign}{:=}

\newcommand{\nobracket}{}
\newcommand{\nocomma}{}
\newcommand{\nosymbol}{}
\newcommand{\tmem}[1]{{\em #1\/}}
\newcommand{\tmmathbf}[1]{\ensuremath{\boldsymbol{#1}}}
\newcommand{\tmop}[1]{\ensuremath{\operatorname{#1}}}
\newcommand{\tmrsub}[1]{\ensuremath{_{\textrm{#1}}}}
\newcommand{\tmrsup}[1]{\textsuperscript{#1}}

\numberwithin{equation}{section}

\def\norm#1.#2.{\lVert#1\rVert_{#2}}
\def\Norm#1.#2.{\bigl\lVert#1\bigr\rVert_{#2}}
\def\NOrm#1.#2.{\Bigl\lVert#1\Bigr\rVert_{#2}}
\def\NORm#1.#2.{\biggl\lVert#1\biggr\rVert_{#2}}
\def\NORM#1.#2.{\Biggl\lVert#1\Biggr\rVert_{#2}}

\def\ip#1,#2,{\langle #1,#2\rangle}
\def\Ip#1,#2,{\bigl\langle#1,#2\bigr\rangle}
\def\IP#1,#2,{\Bigl\langle#1,#2\Bigr\rangle}

\def\mid{\,:\,}

\def\abs#1{\lvert#1\rvert}

\def\XXint#1#2#3{{\setbox0=\hbox{$#1{#2#3}{\int}$}
     \vcenter{\hbox{$#2#3$}}\kern-.5\wd0}}

\def\eqdef{\stackrel{\mathrm{def}}{{}={}}}

\DeclareFontFamily{U}{wncy}{}
\DeclareFontShape{U}{wncy}{m}{n}{<->wncyr10}{}
\DeclareSymbolFont{mcy}{U}{wncy}{m}{n}
\DeclareMathSymbol{\Sh}{\mathord}{mcy}{"58}  

\begin{document}

\title {A Lower Bound Criterion for Iterated Commutators}
\author[L. Dalenc \and S. Petermichl]{Laurent Dalenc$^1$\and Stefanie Petermichl$^2$}

\address{Laurent Dalenc\\ Universit\'e Paul Sabatier\\}

\thanks{$1.$ Research supported in part by an ANR Grant.}

\address{Stefanie Petermichl\\ Universit\'e Paul Sabatier\\}

\thanks{$2.$ Research supported in part by ANR grant. The author is a junior member of IUF.}

\maketitle

\begin{abstract}
  We consider iterated commutators of multiplication by a symbol function $b$
  and smooth Calderon Zygmund operators, described by Fourier multipliers of
  homogeneity 0. We establish a criterion for a collection of symbols so that
  the corresponding Calderon Zygmund operators characterize product
  $\tmop{BMO}$ by means of iterated commutators. We therefore extend, in part,
  the line of one-parameter results following the work of Uchiyama and Li as
  well as the result in several parameters, concerning commutators with
  Riesz transforms by Lacey, Petermichl, Pipher, Wick.
\end{abstract}


\section{Introduction}

\label{s1}A classical result of Nehari {\cite{Ne}} \ shows that a Hankel
operator with antianalytic symbol $b$ is bounded if and only if the symbol
belongs to $\tmop{BMO}$. This theorem has an equivalent formulation by means
of commutators of a symbol function $b$ and the Hilbert transform, as the
latter are a combination of orthogonal Hankel operators. Nehari's result leans
on analytic structure in several crucial ways: the classical
factorization result for $H^1$ functions on the disk and  the
fact that the Hilbert transform is a Fourier projection operator.

The classical text of Coifman, Rochberg and Weiss {\cite{CRW}} \ extended the
one-parameter theory to real analysis in the sense that the Hilbert transforms
were replaced by Riesz transforms. In their text, they obtained sufficiency,
i.e. that a $\tmop{BMO}$ symbol $b$ yields an $L^2$ bounded commutator for
certain more general, convolution type singular integral operators. For
necessity, they showed that the collection of Riesz transforms was
representative enough. This is quite natural, in the view of the definition of
$H^1$ requiring Riesz transforms being back in $L^1$ as well as the 
Fefferman-Stein decomposition of $\tmop{BMO}$ using Riesz kernels.

Uchiyama {\cite{Ufeffstein}} revisited said decomposition, with a very
technical but constructive proof. It remarkably replaced the class of Riesz transforms 
by  more general classes of kernel
operators obeying a certain point separation criterion for their Fourier
multiplier symbols. See also \cite{Uhardychar} and \cite{Uhankel} for more natural questions in this direction. 
Li \cite{li} used a criterion similar to Uchiyama's,
to show that it was also a sufficiently representative class to characterize
$\tmop{BMO}$ by means of commutators.

All of these results date back to the 70s, 80s and 90s and consider $H^1$
spaces in one parameter and simple, i.e. non-iterated commutators.

It is well known that the product theory and with it the product $\tmop{BMO}$
space, as identified by Chang and Fefferman {\cite{CF1}}, \cite{CF2}  have more
complicated structure. We remind of Carleson's interesting example {\cite{C}}
\ illustrating this difference. The techniques to tackle the analogs of the
above questions in several parameters are very different and have brought,
with the works of Lacey and his collaborators, valuable new insight and use to
existing theories, for example in the interpretation of Journe's lemma in
combination with Carleson's example.

Ferguson and Lacey proved in {\cite{FL}} \ that the iterated commutator of the
Hilbert transform and multiplication by a symbol $b$ characterize
$\tmop{BMO}$, and with it, they proved the equivalent weak factorization
result for $H^1$ on the bidisk. Lacey and Terwilliger extended this result to
an arbitrary number of iterates in {\cite{LT}}, requiring thus, among others,
a refinement of Pipher's iterated multi-parameter version of Journ\'e's lemma.
The real variable analog, the result of Coifman, Rochberg and Weiss
{\cite{CRW}} using Riesz transforms instead of Hilbert transforms, was
extended to the multi parameter setting in {\cite{LPPW}}. In this current
paper, we extend in part, the direction of Uchiyama and Lee to several
parameters. We formulate a sufficient condition on a family of CZOs, so that
their iterated commutators characterize $\tmop{BMO}$:

For vectors $\vec{d} = (d_1, ..., d_t) \in \mathbb{N}^t$, we consider product
spaces $$\mathbb{R}^{\vec{d}} = \mathbb{R}^{d_1} \times ... \times
\mathbb{R}^{d_t}.$$ For each $1 \le s \le t$, we have a collection of Calderon
Zygmund operators $\mathcal{T}_s = \{T_{s, 1}, ..., T_{s, n_s} \}$, whose
kernels are homogeneous of degree $- d_s$, with Fourier multiplier symbols
represented by $\theta_{s, k_s} \in \mathcal{C}^{\infty} ( \mathcal{S}^{d_s -
1})$ that are in turn homogeneous of degree 0. For appropriate functions $f$
and symbols $b$ we consider the family of iterated commutators
\[ C_{\vec{k}} (b, \cdot) = [T_{1, k_1} [... [T_{t, k_t}, M_b] ...]] . \]
Here $1 \leq s \leq t, \vec{k} = (k_1, ..., k_t), 0 \leq k_s \leq n_s$ and
$T_{s, k_s}$ denotes the $k_s$th choice of CZO in the family $\mathcal{T}_s$
acting in the $s$th variable.

We impose the following restrictions on the classes $\mathcal{T}_s$ for each
parameter $s$ separately, easiest formulated in terms of their symbols:
\begin{itemize}
  \item $\forall x \neq y \in \mathbb{S}^{d_s - 1} \; \exists \; \theta_{s, i}$ so
  that $\theta_{s, i} ( x) \neq \theta_{s, i} ( y)$
  
  (full point separation on the sphere)
  
  \item $\forall$ $x \in \mathbb{S}^{d_s - 1} \; \forall t \tmop{tangent}
  \tmop{to} \mathbb{S}^{d_s - 1} \tmop{in} x \; \exists \; i \tmop{so} \tmop{that}
  \frac{\partial \theta_{s, i}}{\partial t} ( x) \neq 0$
  
  (existence of non-trivial tangential derivatives)
  \end{itemize}
  
  In the case that the kernels $K$ are not real valued, it appears that a last
condition is needed:

 \begin{itemize} 
  \item $\Theta_s $ is closed under complex conjugation
\end{itemize}

Infinite sets $\mathcal{T}_s$ are also included in our theorem at no additional cost. 

\begin{example}
  It is easy to check that the family of Riesz transforms in
  $\mathbb{R}^{d_s}$ satisfies these properties.
\end{example}

\begin{example}
  It is also not hard to check that the family of all rotations of any one
  smooth, dilation and translation invariant CZO $T$ with a discontinuity in 0 of its symbol in any
  given direction has these properties. Precisely we mean an operator $T$ that has a smooth
  symbol $m$ that is homogeneous of degree zero with the property that there
  exists $\xi \in \mathcal{S}^{d_s - 1}$ such that $m ( \xi) \neq m ( - \xi)$ (or, more generally, even any non-constant symbol $m$!).  Notice that in many cases, such as when we choose $T$ to be the first Riesz
  transform, a small number of rotations are sufficient to make up a family
  with the required properties. 
\end{example}

\begin{theorem}
  Under the conditions above on the classes $\mathcal{T}_s$, there exist
  constants $C_1, C_2 > 0$ so that $\forall b \in \tmop{BMO} (
  \mathbb{R}^{\vec{d}})$
  \[ C_1 || b ||_{BMO} \leq \sup_{0 \leq k_s \leq n_s} || [T_{1, k_1} [...
     [T_{t, k_t}, M_b] ...]] ||_2 \leq C_2 || b ||_{BMO} \]
  where we mean the product $\tmop{BMO}$ norm according to Chang and
  Fefferman. $T_{s, k_s}$ denotes the $k_s$th choice of CZO in the family
  $\mathcal{T}_s$ acting in the $s$th variable.
\end{theorem}

It is well known, that theorems of this form have an equivalent formulation in
the language of weak factorization of Hardy spaces. For $\vec{k}$ a vector
with $1 \le k_s \le d_s$ and $1 \le s \le t$, let us denote by $\Pi_{\vec{k}}$
the bilinear operator obtained by unwinding the commutator:
\[ \langle C_{\vec{k}} (b, f), g \rangle_{L^2} = \langle b, \Pi_{\vec{k}} (f,
   g) \rangle_{L^2} . \]
The operator $\Pi_{\vec{k}}$ can be expressed as linear combination of
iterates of CZOs $T_{s, k_s}$ (and their adjoints), applied to $f, g$.

Using the notation
\[ \|f\|_{L^2 \ast L^2} = \inf \left\{ \sum_{\vec{k}} \sum_j \|
   \phi_j^{\vec{k}} \|_2 \| \psi_j^{\vec{k}} \|_2 \right\} \]
where the infimum runs over all possible decompositions of $f = \sum_{\vec{k}}
\sum_j \Pi_{\vec{k}} (\phi_j^{\vec{k}}, \psi_j^{\vec{k}})$. With the help of
the relevant commutator theorem, it is an exercise in duality to see the
following:

\begin{theorem}
  We have $H^1 ( \mathbb{R}^{\vec{d}}) = L^2 \ast L^2$. For any $f \in H^1 (
  \mathbb{R}^{\vec{d}})$ there exist sequences $\phi_j^{\vec{k}},
  \psi_j^{\vec{k}} \in L^2$ such that $f = \sum_{\vec{k}} \sum_j \Pi_{\vec{k}}
  (\phi_j^{\vec{j}}, \psi_j^{\vec{k}})$ with $\|f\|_{H^1} \sim \sum_{\vec{k}}
  \sum_j \| \phi_j^{\vec{k}} \|_2 \| \psi_j^{\vec{k}} \|_2$.
\end{theorem}

In this text we prefer the language of commutators in terms of upper
(sufficiency) and lower (necessity) bounds.

Our proof follows the machinery developed by Lacey and collaborators in \cite{FL},  \cite{LT},  \cite{LPPW}. In particular, we refine a strategy from \cite{LPPW}, to pass from the complex variable case and the Hilbert transform to the real variable  and Riesz transform case. The Fourier multipliers of the Riesz transforms are very special - monomials on the sphere. We establish such a passage for much more general multiplier operators.

It seems not possible to use any previously proved  characterization theorems directly. We can however reuse some of the general strategy and in particular, we manage to 
'black box'  the very technical wavelet support and paraproduct estimates found in different versions in previous works. In \cite{LPPW}, this part appears to be the most streamlined and is general enough to apply to our situation.

\section{A Brief Review of Multi-Parameter Theory}

\subsection{Wavelets in Higher Dimensions and Several
Parameters}\label{s.wavelet}

We will use the following dilation and translation operators on $\mathbb{R}^d$
\begin{align}
  \label{e.trans}\mathrm{Tr}\tmrsub{y}f(x) &
  {\assign}f(x-y){\hspace{0.25em}},{\hspace{1em}}y{\in}\mathbb{R}\tmrsup{d},\\
  \label{e.dil}\mathrm{Dil}\tmrsub{a}\tmrsup{(p)}f(x) &
  {\assign}a\tmrsup{-d/p}f(x/a){\hspace{0.25em}},{\hspace{1em}}a>0{\hspace{0.25em}},0<p{\le}{\infty}{\hspace{0.25em}}.
\end{align}
These will also be applied to sets, in an obvious fashion, in the case of $p =
\infty$.

By the {\tmem{($d$ dimensional) dyadic grid}} in $\mathbb{R}^d$ we mean the
collection of cubes
\[ \mathcal{D}_d \assign \left\{ j 2^k + [0, 2^k)^d \mid j \in \mathbb{Z}^d
   \hspace{0.25em}, k \in \mathbb{Z} \right\} . \]
   An elementary example of a wavelet system is the Haar system generated by  $h = -
\mathbf{1}_{(0, 1 / 2)} + \mathbf{1}_{(1 / 2, 1)}$ and $W=\mathbf{1}_{(0, 1)}.$ 
The principle requirement is that the functions $\{ \mathrm{Tr}_{c
(I)} \mathrm{Dil}^{(2)}_I w \mid I \in \mathcal{D}_1 \}$ form an orthonormal
basis for $L^2 ( \mathbb{R})$. 

The wavelet in this text should be thought of {\tmem{Meyer wavelet}} though, 
due to its extraordinary Fourier support properties. Although not explicit in this text, we 
borrow certain technical estimates that make decisive use of this feature of the Meyer wavelet. 

For $\varepsilon \in \{0, 1\}$, set $w^0 = w$ and $w^1 = W$, the superscript
\tmrsup{$0$} denoting that `the function has mean $0$,' while a superscript
\tmrsup{$1$} denotes that `the function is an $L^2$ normalized indicator
function.' In one dimension, for an interval $I$, set
\[ w^{\varepsilon}_I \assign \mathrm{Tr}_{c (I)} \mathrm{Dil}_{|I|}^{(2)}
   w^{\varepsilon} \hspace{0.25em} . \]
Multiresolution wavelets, such as the Haar or the Meyer wavelet have the
useful identity
\begin{equation}
  \label{e.FATHER} \sum_{I \supsetneq J} \langle f, w_I \rangle w_I = \langle
  f, w^1_J \rangle w^1_J \hspace{0.25em},
\end{equation}
The passage from $\mathbb{R}$ to $\mathbb{R}^d$ consists of a product of $d$
wavelets associated to intervals of the same size, so that the resulting
wavelet is associated to a cube.

Let $\sigma_d \assign \{0, 1\}^d - \{ \vec{1} \}$, which we refer to
as {\tmem{signatures}}. In $d$ dimensions, for a cube $Q$ with side $|I|$,
i.e., $Q = I_1 \times \cdots \times I_d$, and a choice of $\varepsilon \in
\sigma_d$, set
\[ w^{\varepsilon}_Q (x_1, \ldots, x_d) \assign \prod_{j = 1}^d
   w_{I_j}^{\varepsilon_j} (x_j) . \]
It is then the case that the collection of functions
\[ \mathrm{Wavelet}_{\mathcal{D}_d} \assign \{w_Q^{\varepsilon} \mid Q \in
   \mathcal{D}_d \hspace{0.25em}, \varepsilon \in \sigma_d \} \]
form a wavelet basis for $L^p ( \mathbb{R}^d)$ for any choice of $d$
dimensional dyadic grid $\mathcal{D}_d$. Here, we are using the notation
$\vec{1} = (1, \ldots, 1)$. 

The passage to the tensor product setting, $\mathbb{R}^{\vec{d}} =
\mathbb{R}^{d_1} \times ... \times \mathbb{R}^{d_t}$ consists of a product of
$t$ wavelets associated to cubes of possibly different size, so that the
resulting wavelet is associated to a rectangle.

For a vector $\vec{d} = (d_1, \ldots, d_t)$, and $1 \le s \le t$, let
$\mathcal{D}_{d_s}$ be a choice of $d_s$ dimensional dyadic grid, and let
\[ \mathcal{D}_{\vec{d}} = \otimes_{s = 1}^t \mathcal{D}_{d_s} \hspace{0.25em}
   . \]
Also, let $\sigma_{\vec{d}} \assign \{ \vec{\varepsilon} =
(\varepsilon_1, \ldots, \varepsilon_t) : \varepsilon_s \in
\sigma_{d_s} \}$. Note that each $\varepsilon_s$ is a vector, and so
$\vec{\varepsilon}$ is a `vector of vectors'. For a rectangle $R = Q_1 \times
\cdots \times Q_t$, being a product of cubes of possibly different dimensions,
and a choice of vectors $\vec{\varepsilon} \in \sigma_{\vec{d}}$ set
\[ w_R^{\vec{\varepsilon}} (x_1, \ldots, x_t) = \prod_{s = 1}^t
   w_{Q_s}^{\varepsilon_s} (x_s) . \]
These are the appropriate functions and bases to analyze multiparameter
paraproducts and commutators.

So the collection of wavalets associated to a dyadic grid in the product
setting $\mathcal{D}_{\vec{d}}$ is
\[ \left\{ w_R^{\vec{\varepsilon}} \mid R \in \mathcal{D}_{\vec{d}}
   \hspace{0.25em}, \vec{\varepsilon} \in \sigma_{\vec{d}} \right\}
   \hspace{0.25em} . \]
This is a basis in $L^p ( \mathbb{R}^{\vec{d}})$.

\subsection{Chang--Fefferman $\tmop{BMO}$}

\label{s.cfBMO}

Let us describe product Hardy space theory. By this, we mean the
Hardy spaces associated with domains like $\otimes_{s = 1}^t
\mathbb{R}^{d_s}$.

The Hardy space $H^1 ( \mathbb{R}^d)$ denotes the class of
functions with the norm
\[ \sum_{j = 0}^d \|R_j f \|_1 \]
where $R_j$ denotes the $j$th Riesz transform. We adopt the
convention that $R_0$, the $0$th Riesz transform, is the identity. This space
is invariant under the one parameter family of isotropic dilations, while $H^1
( \mathbb{R}^{\vec{d}})$ is invariant under dilations of each coordinate
separately. This invariance under a $t$ parameter family of
dilations gave rise to the term `multi parameter' theory.

The product space $H^1 ( \mathbb{R}^{\vec{d}})$ has a variety of equivalent
norms, in terms of square functions, (strong) maximal functions and Riesz transforms.

The dual of the real Hardy space is $$H^1 ( \mathbb{R}^{\vec{d}})^{\ast} =
\text{\tmop{BMO}} ( \mathbb{R}^{\vec{d}}),$$ the $t$--fold product BMO space.
It is a Theorem of Chang and Fefferman {\cite{CF2}}  that this space
has a characterization in terms of a product Carleson measure.

Define
\begin{equation}
  \label{e.BMOdef} \lVert b \rVert_{\text{\tmop{BMO}} ( \mathbb{R}^{\vec{d}})}
  \assign \sup_{U \subset \mathbb{R}^{\vec{d}}} \left[ |U|^{- 1} \sum_{R
  \subset U} \sum_{\vec{\varepsilon} \in \sigma_{\vec{d}}} | \langle
  b, w_R^{\vec{\varepsilon}} \rangle |^2 \right]^{1 / 2} .
\end{equation}
Here the supremum is taken over all open subsets $U \subset
\mathbb{R}^{\vec{d}}$ with finite measure, and we use a wavelet basis
$w_R^{\vec{\varepsilon}}$.

\begin{theorem}(Chang, Fefferman)
\label{t.changfefferman} We have
  the equivalence of norms
  \[ \|b\|_{(H^1 ( \mathbb{R}^{\vec{d}}))^{\ast}} \approx
     \|b\|_{BMO ( \mathbb{R}^{\vec{d}})} \]
  That is, $BMO( \mathbb{R}^{\vec{d}})$ is the dual to $H^1
  ( \mathbb{R}^{\vec{d}})$.
\end{theorem}

Notice that this space $\tmop{BMO}$ is invariant under a $t$-parameter family of dilations. Here
the dilations are isotropic in each parameter separately. This fact is also represented 
by the choice of our wavelet system.

\subsection{Journ\'e's Lemma}

Notice that the supremum in the wavelet definition of BMO runs over open sets of finite measure. This supremum restricted just to rectangles gives the definition of the larger rectangular BMO. There is a substantial geometric difference: the maximal dyadic sub-rectangles of any arbitrary rectangle are disjoint while those maximal dyadic sub-rectangles in open sets are not necessarily comparable by inclusion. It is in part due to this difference that, in the same way as in \cite{FL}, a geometric lemma by Journ\'e \cite{J} involving rectangles in the plane, particularly useful in handling collections of rectangles not comparable by inclusion, comes into play. It was first observed by Ferguson and Lacey that Journ\'e's lemma could be improved to partially compare rectangular BMO and product BMO of two parameters.

A $n$-dimensional version of Journ\'e's  original lemma is due to Pipher \cite{Pi} and makes use of iterations. This is the reason why we are going to have to replace the rectangular BMO space by another version of BMO that allows us to induct on the number of parameters in our commutator and therefore make use of the iterated nature of Journ\'e's lemma in more than two parameters. This idea was first used in \cite{LT}. 

Say that a collection of rectangles $\mathcal{U} \subset
\mathcal{D}_{\vec{d}}$ has $t - 1$ parameters \ if and only if there is a
choice of coordinate $s$ so that for all $R, R' \in \mathcal{U}$ we have $Q_s
= Q_s'$, that is the $s$th coordinate of the rectangles are all one fixed
$d_s$ dimensional cube.

We then define
\[ \| f \|_{\tmop{BMO}_{- 1} ( \mathbb{R}^{\vec{d}})} = \sup_{\mathcal{U}
   \tmop{has} t - 1 \tmop{parameters}} \left( | \tmop{sh} ( \mathcal{U}) |^{-
   1} \sum_{\vec{\varepsilon}} \sum_{R \in \mathcal{U}} | \langle f,
   w_R^{\vec{\varepsilon}} \rangle |^2 \right)^{1 / 2} \]

In this notation, a collection of rectangles has a shadow given by $\tmop{sh} ( \mathcal{U}) =
\cup \{ R : R \in \mathcal{U} \}$. The $- 1$ subscript is used to indicate
that we have `reduced by one parameter' in the definition.  The reader may be more familiar with the
rectangular $\tmop{BMO}$ space mentioned above. In two parameters, the space $\tmop{BMO}_{-
1}$ is larger than rectangular $\tmop{BMO}$.

Carleson produced examples of functions which acted as linear functionals
on \ $H^1 ( \mathbb{R}^{\vec{d}})$ with norm one, yet had arbitrarily small
 rectangular $\tmop{BMO}$ norm (and hence arbitrarily small $\tmop{BMO}_{-
1}$ norm). 

Here is the precise version of the above mentioned refinement of Journ\'e's lemma. It permits us, with certain restrictions and by inducing a damping factor, to control the
$\tmop{BMO}$ norm by the $\tmop{BMO}_{- 1} $norm. 
\begin{lemma}
  Let $\mathcal{U}$ be a collection of rectangles of finite shadow. For any $a > 0$, we can
  construct $V \supset \tmop{sh} ( \mathcal{U})$ together with a function
  $E : \mathcal{U} \rightarrow [ 1 \nocomma, \infty] $
  so that
  $E ( R) \cdot R \subset V $
  for all $R \in \mathcal{U}$,
  $ | V | < ( 1 + a) | \tmop{sh} ( \mathcal{U}) |, $ and last that 
  \[ \left\| \sum_{\vec{\varepsilon}} \sum_{R \in \mathcal{U}} E ( R)^{- C}
     \langle b, w_R^{\vec{\varepsilon}} \rangle w_R^{\vec{\varepsilon}}
     \right\|_{\tmop{BMO}} \leq K_a \| b \|_{\tmop{BMO}_{- 1}} . \]
  Here $C$ depends only on $\vec{d}$ and $K_a$ on $a$ and $\vec{d}$.
    
\end{lemma}

A good and more complete reference on the subject is \cite{varjourne}.

\subsection{Remarks on the Upper Bound}

We are going to assume that $ K$ is a smooth Calder\'on--Zygmund convolution 
kernel on $ \mathbb R ^{d} \times 
\mathbb R ^{d}$.  This means that the kernel is a distribution that 
satisfies the estimates below for $ x\neq y$
\begin{equation} \label{e.CZ}
\begin{split}
\abs{ \nabla ^{j} K (y)} &\le N  \abs{ y} ^{-d-j}\,, \quad j=0,1,2, \dotsc, d+1\,. 
\\
\norm \widehat K . L ^{\infty } . &\le N  \,. 
\end{split}
\end{equation}
The first  estimate combines the standard size and 
smoothness estimate.
The last assumption is equivalent 
to assuming that the operator defined on Schwartz functions by 
\begin{equation*}
\operatorname T _{K} f (x) \eqdef \int K (x-y) f (y)\; dy
\end{equation*}
extends to a bounded operator on $ L^2 (\mathbb R ^{d})$.  
 
 If $ K_1,\dotsc,K_t$ is a sequence of Calder\'on--Zygmund kernels, 
 with $ K_s$ defined on $ \mathbb R ^{d_s} \times \mathbb R ^{d_s}$.  
 It is not obvious that the corresponding tensor product operator 
 \begin{equation*}
T _{K_1} \otimes \cdots \otimes T _{K_t}
\end{equation*}
is a bounded operator on $ L^p (\mathbb R ^{\vec d})$.  This is a consequence of  
multi-parameter Calder\'on--Zygmund theory.

By \cite{LPPW}, theorem (5.3), multi-parameter commutators are bounded operators if the symbol belongs to $\tmop{BMO}$:

\begin{theorem}\label{t.upper}  For $ 1<p<\infty $, 
\begin{equation}\label{e.BMOupper}
\norm  [ \operatorname T _{K_1}, 
\cdots [ \operatorname T _{K_t},\operatorname M_b]
\cdots ] .p\to p. \lesssim  \norm b.\textup{BMO}.\,.
\end{equation}
By $ \textup{BMO}$, we mean Chang--Fefferman $ \textup{BMO}$. 
The implied constant depends upon the vector $ \vec d$, and the 
$ \operatorname T _{K_s}$. 
\end{theorem}

The Calderon-Zygmund operators we are concerned about in this text are assumed to have 'infinite' smoothness in the sense of the estimates on the kernel in \ref{e.CZ} and are therefore included in the result above.

The rest of the paper is dedicated to establishing a lower estimate of our commutators 
by means of product $\tmop{BMO}$. We are going to follow the iteration strategy in \cite{FL}, 
\cite{LT}, \cite{LPPW}. The one-dimensional case is very special: the Hilbert transform is both 
Calderon-Zygmund as well as half space Fourier projection operator. We have lost this feature 
in higher dimensions, but it motivates the use of CZOs close to projection operators, such as in 
\cite{LPPW}.

\section{Cone operators}\label{s.cones}

In dimension $d \geqslant 2$, a cone $C \subset$ $\mathbb{R}^d$ is given by
the data $( \xi, Q)$ where $\xi \in \mathbb{R}^d$ is the direction of the
cone and the cube $Q \subset \xi^{\bot}$ centered
at the origin is its aperture. The cone consists of all vectors $\theta$ that
take the form $( \theta_{\xi} \xi \nocomma, \theta^{\bot}) \tmop{where}
\theta_{\xi} = \theta . \xi$ and $\theta^{\bot} \in \theta_{\xi} Q.$ By
$\lambda C$ we mean the dilated cone with data $( \xi, \lambda Q)$.

Given a cone $C$, we consider its Fourier projection operator defined via
$\widehat{P_C} f =\tmmathbf{1}_C \hat{f} .$ Due to the fact that the apertures
are cubes, such operators are combinations of Fourier projections onto half
spaces and as such admit uniform $L^p$ bounds. For a given
cone $D$ we consider a smooth Calderon-Zygmund operator $T_D$ with a kernel $K_D
$ whose Fourier symbol $\widehat{K_D} \in C^{\infty}$ and satisfies the
estimate $\tmmathbf{1}_D \leqslant \widehat{K_D} \leqslant \tmmathbf{1}_{( 1 +
\tau) D}$.

\begin{remark}\label{r.lp}
The derivatives of the symbols $\widehat{K_D}$ increase
with the aperture of the cones. In the course of the proof it will be
important that the $L^p$ bounds of operators $T_D$ do not grow with the
aperture of the cones. We thank the special nature of the cone operators and their closeness to half plane projections for this fact. By a rotation argument, we may assume that the cone $D$
has direction $x_1$. There exists a smoothed symbol $m$ of the sort described,
so that higher derivatives in consecutive directions $x_2, \ldots ., x_n$ are
controlled independently of the aperture. In the remaining variable, $x_1$, the derivatives grow with the aperture, but we
control total variation of the derivatives in $x_2, \ldots ., x_n$. In doing so and carefully reading the Marcinkiewicz multiplier theorem, it provides us with $L^p$ bounds independent of the aperture. The
details are left to the interested reader. We refer to \cite{G} page 363 for a detailed statement of this theorem.
\end{remark}

\subsection{Selection of a Representative Class of Cones}

Following the idea in \cite{LPPW}, we select classes of cones that are going to 
give us a certain auxiliary lower bound. We felt the need to refine this process, 
which is necessary due to the fact that we consider more general classes of CZOs instead 
of just the class of Riesz transforms.

Let $b$ be our $\tmop{BMO}$ function that we normalize to have norm 1. \ Let
$U$ be the open set that gives us the supremum in the $\tmop{BMO}$ norm of $b$
and denote by $\mathcal{U}$ the collection of rectangles $R \subset U$. Let us
renormalize, by an appropriate dilation, the size of the set $\tmop{sh}
(\mathcal{U})$ to be comparable to 1. Let $\beta = P_{\mathcal{U}} b$, the
wavelet projection onto those wavelets adapted to rectangles in the class
$\mathcal{U}$.

Given a cone $C$ with data $(\xi, Q)$. We denote by $H_C$ the half plane
projection that corresponds to the direction $\xi$, the convolution operator
whose symbol is $\chi_{(0, \infty)}  (\xi \cdot \theta)$. Recall that $T_C$
denotes the CZO adapted to the cone and $P_C$ the Fourier projection
associated to the cone. Given a vector of cones $\vec{C} = (C_s)_{1 \le s \le
t}$ we denote by $H_{\vec{C}}, T_{\vec{C}}, P_{\vec{C}}$ their tensor
products.

\begin{lemma}\label{l.cone}
  Let $b$ be the set of all $\tmop{BMO}$ functions normalized as above. For all
  such $b$, let $U \nocomma, \mathcal{U}, \beta$ be as above. For any $\kappa
  > 0$ we can select a finite set of pairs $( \vec{D}, \vec{C})$of vectors of
  cones $\vec{D} = (D_s)_{1 \le s \le t}$ where $D_s \subset \mathbb{R}^{d_s}$
  with data $(\xi_s, Q_s)$ and $C_s \subset \mathbb{R}^{d_s}, 1 \le s \le t$
  with data $(\xi'_s, Q'_s)$ so that for each $\beta$ there is a pair $(
  \vec{D}, \vec{C})$ with the following properties.
  \begin{enumerate}
    \item $D_s \subset C_s$
    
    \item $\|T_{\vec{D}} \beta \|_2 \ge 4^{- t}$
    
    \item $\|(H_{\vec{D}} - T_{\vec{D}}) \beta \|_4 \le \kappa$
    
    \item $\|(H_{\vec{C}} - P_{\vec{C}}) |T_{\vec{D}} \beta |^2 \|_2 \le
    \kappa$
  \end{enumerate}
\end{lemma}

\begin{proof}
  We first select a finite collection of cones $D_s$. Let us for the moment fix
  $b$. Let $\eta$ be a small positive number to be determined later. It will be
  in relation with the aperture of the cones: given $\eta$, the aperture $Q_s$
  is chosen large enough so that
  \[ \mathbb{P}  (D_s \cap \mathbb{S}^{d_s - 1} | \nobracket \mathbb{S}^{d_s -
     1}) \ge \frac{1}{2} - \eta . \]
  We consider random rotations $D_s^{\phi_s}$ of $D_s$ and write
  $\vec{D}^{\phi}$ for component-wise independent rotation.
  
  Averaging the $L^2$ norms gives us
  \[ \mathbb{E} (\|P_{\vec{D}^{\phi}} \beta \|^2_2) = \mathbb{E} (
     \int_{\vec{D}^{\phi}} | \hat{\beta} (\xi) |^2 d \xi) \ge ( \frac{1}{2} -
     \eta)^t \]
  as well as
  \[ \mathbb{E} (\|(H_{\vec{D}^{\phi}} - P_{\vec{D}^{\phi}}) \beta \|^2_2) \le
     \eta^t . \]
  Notice that for all choices of $\phi$, we have
  \[ 0 \le \|T_{\vec{D}^{\phi}} \beta \|_2 \le 1 \]
  as well as
  \[ 0 \le \|(H_{\vec{D}^{\phi}} - T_{\vec{D}^{\phi}}) \beta \|_2 \le 1 \]
  Together, this provides us with the estimates
  \[ \mathbb{P}  (\|T_{\vec{D}^{\phi}} \beta \|_2 \ge 4^{- t}) \ge \frac{4^t 
     ( \frac{1}{2} - \eta)^t - 1}{4^t - 1} \]
  and
  \[ \mathbb{P}  (\|(H_{\vec{D}^{\phi}} - T_{\vec{D}^{\phi}}) \beta \|_2 \ge
     \eta^{- t}) \le \eta^{\frac{t}{2}} \]
  Since $\lim_{\eta \to 0}  \frac{4^t  ( \frac{1}{2} - \eta)^t - 1}{4^t - 1} =
  \frac{1}{2^t - 1}$ and $\lim_{\eta \to 0} 1 - \eta^{\frac{t}{2}} = 1$, the
  sum of the above probabilities exceeds 1 for small enough $\eta$. In this case we are sure to be able
  to select directions so that
  \[ \|T_{\vec{D}^{\phi}} \beta \|_2 \ge 4^{- t} \]
  and
  \[ \|(H_{\vec{D}^{\phi}} - T_{\vec{D}^{\phi}}) \beta \|_2 \le \eta^{- t / 2}.
  \]

  We have half plane projection operators $H_D$ and CZOs $T_D$ that have,
  according to remark \ref{r.lp} above, uniform $L^p$ bounds. Also remember that
  $\beta$ is normalized in $L^2$ as well as in $\tmop{BMO}$. We therefore have
  uniform $L^8$ bounds: $\|(H_{\vec{D}^{\phi}} - T_{\vec{D}^{\phi}}) \beta
  \|_8 \le K$ where the constant $K$ neither depends on the aperture nor the
  direction of the cones.
  
  By interpolation we get $\|(H_{\vec{D}^{\phi}} - T_{\vec{D}^{\phi}}) \beta
  \|_4 \lesssim \eta^{- t / 6}$. We choose $\eta$ small enough so that both
  the above inequalities hold as well as $\eta^{- t / 6} < \kappa$.
  
  We have seen that there exists a fixed $\eta$ so that for each $b$ the set
  $b( \eta) \subset \mathbb{S}^{d - 1}$ of admissible directions $\xi$ is
  not empty. Notice that $b ( \eta) \subset b ( \eta / 2)$. Furthermore, there
  exists $r ( \eta)$ so that the ball $B ( \xi, r ( \eta)) \cap
  \mathbb{S}^{d- 1} \subset b ( \eta / 2) $for all $\xi \in b (
  \eta) .$ So by increasing the aperture, a dense enough finite sample set of directions will therefore
  provide an admissible direction for all appropriately normalized
  $\tmop{BMO}$ functions $b$.

  We turn to the selection of cones $C_s$, keeping in mind that cones $D_s$
  have already been chosen. Due to uniform $L^4$ estimates of
  $T_{\vec{D}^{\phi}} $ we see that $\|| \gamma |^2 \|_2 \le K$ for some
  universal $K$. So, in particular, for any vector of cones $\vec{C}$, we have
  $\|(H_{\vec{C}} - P_{\vec{C}}) | \gamma |^2 \|_2 \le K$.
  
  Take $\varsigma < \eta / 2$ a small positive number. Choosing the aperture
  of the cones $C_s$ large enough so that
  \[ \mathbb{P}  (C_s \cap \mathbb{S}^{d_s - 1} | \nobracket \mathbb{S}^{d_s -
     1}) \ge \frac{1}{2} - \varsigma \]
  gives us the estimate
  \[ \mathbb{E} \|(H_{\vec{C}^{\phi}} - P_{\vec{C}^{\phi}}) |
     \gamma |^2 \|_2 \le K \varsigma^t \]
  Similarly to above,
  \[ \mathbb{P}  (\|(H_{\vec{C}^{\phi}} - P_{\vec{C}^{\phi}}) | \gamma |^2
     \|_2 \ge K \varsigma^{t / 2}) \le \varsigma^{t / 2} \]
 
  If $D_s = ( \xi_s,
  Q_s)$ let $E_{\xi_s}$ be the hyperplane perpendicular to $_{} \xi_s$ and
  $H_{\xi_s} \tmop{the} \tmop{corresponding} \tmop{half} \tmop{space}
  \tmop{that} \tmop{contains}$ $D_s$. Let $\alpha = \min \angle ( \xi_1,
  \xi_2) : \xi_1 \in D_s \nocomma, \xi_2 \in E_{\xi_s}$ where $\angle$ denotes
  the angle between vectors. Notice that $\alpha$ only depends upon $\eta$.
  Consider now the circular cone $A_{_{} \xi_s} = \{ \xi : \angle ( \xi,
  \xi_s) < \alpha / 4 \}$. There exists a fixed larger aperture $Q_{s}'$,
  only depending on $\alpha$ so that $( \xi, Q_s') \supset ( \xi_s,
  Q_s)$ whenever $\xi \in D_{\xi_s}$. We are free to choose $\varsigma$ small
  enough so that
  \[ \mathbb{P}  (A_{\xi_s} \cap \mathbb{S}^{d_s - 1} | \nobracket
     \mathbb{S}^{d_s - 1}) \ge \varsigma^{1 / 2} \]
  as well as $K \varsigma^{t / 2} < \kappa$. Since
  \[ \mathbb{P}  (\|(H_{\vec{C}^{\phi}} - P_{\vec{C}^{\phi}}) | \gamma |^2
     \|_2 \ge K \varsigma^{t / 2}) \le \varsigma^{t / 2} \]
  we are sure to find $C_s = ( \xi_s', Q_s')$ with the required properties.
  
  By slightly enlarging the aperture of cones $C_s$ and an argument similar to
  the one above, we obtain a finite collection of cones $C_s$ with the
  required properties.

\end{proof}

We form commutators using  arbitrary cones $C_s = ( \xi_s, Q_s)$.  Let
us define
\[ \|b\|_{\vec{Q}} = \sup \|[T_{C_1}, ... [T_{C_t}, M_b] ...]\|_{2 \to 2} \]
where the supremum is taken over all choices of cone transforms $T_{C_s} =
T_{( \xi_s, Q_s)}$ in which the direction $\xi_s$ varies and the aperture
of the cone is fixed to be $Q_s$ for each parameter $s$ separately. Here $T_{C_s}$ acts
in the $s$th variable. In \cite{LPPW} the following theorem was proven:

\begin{theorem}
  $\|b\|_{\vec{Q}} \sim \| b \|_{\tmop{BMO}}$ with constants depending upon
  the aperture of the cones.
\end{theorem}

We are going to need information that is somewhat more specific. It is
valuable to us to know for which test function, depending on the symbol $b$,
the commutator becomes large. 

\begin{lemma}
  If $\gamma = T_{\vec{D}} \beta$ with cones $\vec{D}, \vec{C}$ chosen as in
  the lemma, then
  \[ \|[T_{C_1}, ... [T_{C_t}, M_b] ...] \bar{\gamma} \|_2 \gtrsim 1. \]
\end{lemma}

The proof of a similar estimate is implicit in \cite{LPPW}, section 7. Although the cones in
our text have somewhat different properties ($D_s$ and $C_s$ do not necessarily share the same direction), 
the pairs $(D_s,C_s)$  were chosen to enable the use of the proof in \cite{LPPW}. We sketch the part of the proof that illustrates the special use of the cone operators.

Let $U$ be the supremal set in the definition of $\tmop{BMO}$ and  $\mathcal{U}$ the corresponding collection of dyadic rectangles with its shadow $sh(\mathcal{U})$. Journ\'e's lemma provides us with a slightly larger set $V$. Let $\mathcal{V}=\{R:R\subset V, R \not\subset sh(\mathcal{U})\}$. Let $\mathcal{W}$ denote the rest of the dyadic rectangles.

We first observe that with $\beta=\mathcal{P}_{\mathcal{U}}b$  and $\gamma= T_{\vec{D}}\beta$, we have $\|[T_{C_1},...[T_{C_t},M_{\beta}]...]\bar{\gamma}\|_2 \gtrsim 1$. Observe that the only non-zero term in this commutator is $T_{C_1}...T_{C_t}(\mathcal{P}_{\mathcal{U}}b)\bar{\gamma}$ since any cone operator falling on $\bar{\gamma}$ is zero. 
Consider now the splitting 
$$T_{\vec{C}}(\gamma+(H_{\vec{C}}-T_{\vec{D}})\beta + (I-H_{\vec{C}}) \beta )\bar{\gamma}.$$
 The last term is zero since $(I-H_{\vec{C}}) \beta $ and $ \bar{\gamma}$ are supported on the same half space away from the cones $\vec{C}$. The second term is small due to the choice of the cone in lemma (\ref{l.cone}). The first term is large and explains the motivation using cone transforms:
$$\norm  T_{\vec{C}}
[ \gamma  \cdot \overline\gamma  ] .2.+\kappa 
\ge \norm 
H _{\vec{C}} 
[\gamma  \cdot \overline\gamma  ] .2.
\gtrsim 
\norm  \overline {\gamma } \cdot \gamma .2.  
= \norm \gamma .4.^2 
 \gtrsim 1\,. 
$$
This follows as the Fourier transform of $ \overline {\gamma } \cdot \gamma $ 
is symmetric with respect to the half planes determined by the cones; 
the last inequality uses the Littlewood--Paley inequalities.

Next, we will see that  $\|[T_{C_1},...[T_{C_t},M_{\mathcal{P}_{\mathcal{V}}b}]...]\bar{\gamma}\|_2 \lesssim \delta_J^{1/4}$.  It is easy to see that
$$
\norm [ \operatorname T _{C_1}, \cdots [T_{C_t},\operatorname M _{ P_{\mathcal V}b}
] \cdots ]\overline\gamma .2.
\lesssim\norm\operatorname P_{\mathcal V}b.4.\norm \gamma  .4.
\lesssim \norm\operatorname P_{\mathcal V}b.4.\,,
$$
where the implied constant depends upon the $ L^4$ 
norms of the Cone transforms. 
But, by Journ\'e's lemma, we have that 
$$
\norm\operatorname P_{\mathcal V}b.2.\leq\delta_J^{1/2}\, ,
\qquad\norm\operatorname P_{\mathcal V}b.\textup{BMO}.\leq 1\,,.
$$
Together they imply
$$
\norm\operatorname P_{\mathcal V}b.4.\leq\delta_J^{1/4}.
$$

For the technical estimate of the last term as well as a more detailed exposition, 
we refer to \cite{LPPW} section 7, proof of 7.9.

We gather the information and are left with the following:

\begin{theorem}
  For each parameter $s$ there exists a {\it finite} collection $\mathcal{C}_s$ of
  cones $C_{s, k_s} = ( \xi_{k_s}, Q_s) \nocomma$ with $1 \leqslant k_s
  \leqslant n_s$ of fixed aperture $Q_s$ so that
  \[ \|b\|_{BMO} \lesssim \sup \|[T_{C_{1, k_1}}, ... [T_{C_{t, k_t}}, M_b]
     ...]\|_{2 \to 2} \lesssim \|b\|_{BMO} \]
  for all BMO functions b. Here the supremum runs over all $C_{s, k_s} \in
  \mathcal{C}_s$.
\end{theorem}

It will be essential for us to approximate symbols of cone operators using
polynomials in members of our given collections of symbols.

\subsection{Approximation of Cones via the Family $\Theta$}\label{l.approximation}

For a fixed parameter, given our family $\Theta$, we wish to approximate the symbol of cone
projection operators by means of polynomials in $\theta_i$. For technical
reasons, we need a very good approximation that controls also the supremum
norm of derivatives of the symbols, say of order $d \nosymbol \nosymbol .$
Nachbin's beautiful theorem \cite{Na} allows us, under certain conditions on the family,
to do so. We state it in the form we are going to need it.

\begin{theorem}
  Let $\mathfrak{M}$ be a compact smooth manifold. Let $B$ be a closed real
  subalgebra of $A \nocomma = ( C^m_{} ( \mathfrak{M}) \nocomma, \tau_m)$
  where  $\tau_m $ is the topology induced by the norm of uniform convergence
  in $C^m \nocomma$. Then $B = A$ if and only if $B$ contains the function
  $1$, $\forall x \neq y \in \mathfrak{M} \exists f \in B$ such that $f ( x)
  \neq f ( y)$ and for every $x \in \mathfrak{M}$ and $0 \neq v \in T_x (
  \mathfrak{M})$ there exists $f \in B$such that $\tmop{df} ( x) ( v) \neq 0$.
\end{theorem}

It is not hard to check that under the additional assumtion that $B$ be closed
under complex conjugation, there is a complex version.

\begin{lemma}
  For a given $d$ dimensional pair of cones $D$ and $C$ as in lemma (\ref{l.cone}), let
  $H_{- \xi_C}, H_{- \xi_D}$ denote the opposing half spaces, respectively.
  Choose a function $h_{C, D} \in C^d ( \mathbb{S}^{d - 1})$ with values between 0 and
  1 such that
  \begin{itemize}
    \item $h_{C,D} ( \xi) = 1 \forall \xi \in C$
    
    \item $h_{C,D} ( \xi) = 0 \forall \xi \in H_{- \xi_C} \cup H_{- \xi_D}$
  \end{itemize}
  \ Given any small $\epsilon > 0$, there exists an operator $F_{D, C}$ with
  symbol $v_{D, C}$, \ that is a polynomial in $\theta_{} \in \Theta_{}$ so
  that $\| v_{D, C} - h_{D, C} \|_{\tau_d} < \epsilon$, where $\| .
  \|_{\tau_d}$ is the norm of uniform convergence in $C^d$.We have universal
  $L^p$ estimates for the associated kernel operators $F_{D, C} : \|F_{D, C}
  \|_p \lesssim K_p$ where this constant is independent of the choice of the
  cone and universal for small $\epsilon$.
\end{lemma}

\begin{proof}

  Thanks to our assumptions, the part concerning the approximations is almost
  clear. Just observe that we may add the identity operator $I$ with
  multiplier $1$ to our collection. That is the collection $\Theta$
  characterizes $\tmop{BMO}$ if and only if $\Theta \cup \{ 1 \}$ does. In the
  case that the kernels are real valued, we did not assume that $\Theta$ be
  closed under complex conjugation. In this case, consider $\Theta \cup
  \bar{\Theta} $ characterizes $\tmop{BMO}$ if and only if $\Theta$ does.
  Observe that if $T_{\theta}$ denotes the CZO associated to the symbol
  $\theta$, then $T_{\theta}^{\ast} = T_{\bar{\theta}}$. Observe also
  that $[ T, b] = [ T^{\ast}, \bar{b}]^{\ast}$. If the kernel $K ( x)$ of $T$
  is real, then $K ( - x)$ is the kernel of $T^{\ast}$. It is easy to verify
  that
  \[ [ T_1, [ T_2^{\ast}, b]] f = [ T_1 [ T_2, b^{( \cdot, -)}]] f^{( \cdot,
     -)} . \]
  Here $f^{( \cdot, -)} ( x, y) = f ( x, - y)$ so $f$ has a sign change in the
  second set of variables. Its obvious generalization holds when more iterates
  and adjoints are present. The $\tmop{BMO}$ and $L^2$ norms are preserved
  under these reflections.
  
  It remains the important point of universal $L^p \tmop{estimates} .$  Thanks to the control on the 
  derivatives granted to
  us by Nachbin's theorem, we may apply a standard multiplier theorem \cite{G} page to obtain uniform $L^p$ bounds. 
  \end{proof}

\section{Lower bound CZO}\label{s.lowerbound}

We induct on the number $t$ of parameters, that is the number of coordinates in $\vec{d} = ( d_1, \ldots .,
d_t)$.We assume that $d_s \geqslant 2$ for all $s$. The case when $d_s = 1$
for some $s$ reduces our choices of admissible operators to the Hilbert
transform. This case is easier and merely complicates notation for us.

The base case $t = 1$ of our induction argument is stronger than what we need
and a theorem by Li:

\begin{theorem}
Let $\mathcal{T}$ be a collection of CZOs, where the following restriction 
is  imposed: the symbols of the $T_i\in \mathcal{T}$ satisfy $\sum | \theta_{i} ( x) -
  \theta_{i} ( - x) | \neq 0$ for all  $x \in \mathbb{S}^{d - 1}$.

  In the case of $t = 1$ for all $d \geqslant 2$ and symbols $b$ on
  $\mathbb{R}^d$ we have
  \[ \| b \|_{\tmop{BMO}} \lesssim \sup_{1 \leqslant k \leqslant n} \| [ M_b,
     T_k] \|_{2 \rightarrow 2} . \]
  Here \ $T_{k}$ denotes the $k$th choice of CZO in the family
  $\mathcal{T}$.

\end{theorem}

We are also going to need the following weaker lower bound in terms of the
$\tmop{BMO}_{- 1} $ norm in terms of iterated commutators using our families of CZOs.

\begin{lemma}\label{l.lower-1}
  Let $t \geqslant 2$. Given classes $\mathcal{T}_s$ of CZOs with the class of their symbols 
  $\Theta_s$. Assume that for each
  parameter $1 \leqslant s \leqslant t$ separately we have
  \begin{enumerate}
    \item $\forall x \neq y \in \mathbb{S}^{d_s - 1} \; \exists \; \theta_{s, i}$
    so that $\theta_{s, i} ( x) \neq \theta_{s, i} ( y)$
    
    \item $\forall$ $x \in \mathbb{S}^{d_s - 1} \; \forall \; t \tmop{tangent}
    \tmop{to} \mathbb{S}^{d_s - 1} \tmop{in} x \; \exists \; i \tmop{so}
    \tmop{that} \frac{\partial \theta_{s, i}}{\partial t} ( x) \neq 0$
  \end{enumerate}
  and assume that under these same conditions the lower bound holds in the
  case of $t - 1$ parameters in terms of product $\tmop{BMO}$. Then we have
  the estimate
  \[ \| b \|_{\tmop{BMO}_{- 1}} \lesssim \sup_{\vec{k}} \| C_{\vec{k}} ( b,
     \cdot) \|_{2 \rightarrow 2}, \]
  where $C_{\vec{k}} ( b, \cdot) =$ $[T_{1, k_1} [... [T_{t, k_t}, M_b] ...]]$.
  Here $1 \leq s \leq t, \vec{k} = (k_1, ..., k_t), 0 \leq k_s \leq n_s$ and
  $T_{s, k_s}$ denotes the $k_s$th choice of CZO in the family $\mathcal{T}_s$
  acting in the $s$th variable.
  \end{lemma}

The proof uses a well established equivalent formulation of commutator estimates and weak
factorization. This argument goes back to Ferguson and Sadosky \cite{FS}. Our case is closest to the proof of lemma (6.3)
in \cite{LPPW}, replacing the collection of Riesz transforms by our families
$\mathcal{T}_s$. We include a sketch for the sake of completeness.

We assume that $ t\ge2$ and use the induction hypothesis to 
establish a lower bound in terms of our $ \textup{BMO}$ norm with $ t-1$ parameters.

\begin{proof}

It is sufficient to demonstrate that the following inequality holds,
\begin{equation}
\label{e.t-1}
\norm b.(L^2 * L^2)^*.\gtrsim\norm b.\textup{BMO}_{-1}.,
\end{equation}
and this will be established, inducting on the number of parameters.  Assume the truth of the Theorem in $t-1$ parameters.

Given a smooth symbol $b(x_1,\ldots x_t)=b(x_1,x')$ of $t$ parameters, 
we assume that $\norm b.\textup{BMO}_{-1}.=1$.  Assume the supremum is achieved by the collection $\mathcal{U}$ of $\mathcal{D}_{\vec d}$ of $t-1$ 
parameters. Say that  the rectangles in $ \mathcal U$ agree in the 
first coordinate, to a fixed cube $ Q\subset \mathbb R ^{d_1}$. 
After normalization, assume that $\abs{Q}=1$ and $\abs{\textnormal{sh}(\mathcal{U})}\approx 1$. 
Then define 
$$
\psi = \sum_{R\in\mathcal{U}}\sum_{\vec\varepsilon\in\textup{Sig}_{\vec d}}
\ip b,w_R^{\vec\varepsilon}, w_R^{\vec\varepsilon}.
$$
Note that $\ip b,\psi,=1$.  To prove the claim, it is then enough to
prove that $\norm \psi .L^2(\mathbb R^{\vec d})* L^2(\mathbb R^{\vec d}).\lesssim 1$.  Observe that $\psi(x)=\psi_1(x_1)\psi'(x')$ and $\psi_1\in H^1(\mathbb R^{d_1})$ with 
$$
\norm \psi_1.H^1(\mathbb R^{d_1}). =1.
$$
To $\psi_1$, apply the one parameter weak factorization of 
$H^1(\mathbb R^{d_1})$ resulting from the one-parameter characterization result of Li. 
There exists functions $f_n^j,g_n^j\in L^2(\mathbb R^{d_1})$, $n\in\mathbb N$, 
$1\leq j_1\leq d_1$, such that
$$
\psi_1=\sum_{n=1}^\infty\sum_{j_1=1}^{d_1}\Pi_{1,j_1}(f_n^{j_1},g_n^{j_1})
$$
where $\Pi_{1,j_1}(p,q):=T_{1,\, j_1}(p) q+pT_{1,\, j_1}(q)$. 
One next sees that $\psi'\in H^1(\otimes_{l=2}^{t}\mathbb R^{d_l})$ 
with norm controlled by a constant.  
By the induction hypothesis in $ t-1$ parameters, in particular that $H^1(\otimes_{l=2}^{t}
\mathbb R^{d_l})=L^2(\otimes_{s=2}^t\mathbb R^{d_s})*
L^2(\otimes_{s=2}^t\mathbb R^{d_s})$, we have $f_m^{\vec{j}},g_m^{\vec{j}}\in 
L^2(\otimes_{s=2}^{t}\mathbb R^{n_s})$ with $m\in\mathbb N$ and $\vec{j}$ a 
vector with $1\leq j_s\leq d_s$ for $s=2,\ldots, t$ such that
$$
\psi'=\sum_{m=1}^{\infty}\sum_{\vec{j}}\Pi_{\vec{j}}  
(f_m^{\vec{j}},g_m^{\vec{j}}),\qquad \sum_{m=1}^{\infty}\sum_{\vec{j}}
\norm f_m^{\vec{j}} .2. \norm g_m^{\vec{j}}.2. \lesssim 1.
$$

This immediately implies (\ref{e.t-1}) since $\psi=\psi_1\psi'$, and we have a weak factorization of $\psi$ with $\norm \psi .L^2(\mathbb R^{\vec d})* L^2(\mathbb R^{\vec d}).\lesssim 1$. 
\end{proof}

We now turn to the induction step in the main theorem, to finish the proof of the lower estimate
in terms of $\tmop{BMO}$ in $t$ parameters. 

\begin{proof}

  We start with any $\tmop{BMO}$ function $b$ so that $\| b \|_{\tmop{BMO} -
  1} < \delta_{- 1}$ is small. Notice that we have no loss of generality here: due to lemma
  (\ref{l.lower-1}), we already have a lower bound for such $b$ where $\| b
  \|_{\tmop{BMO}_{- 1}} \geqslant \delta_{- 1}$.
  
  We normalize the function $b$ as before, find the function $\beta$ and
  obtain cones $D_s$, the function $\gamma$ and cones $C_s$ according to 
  lemma (\ref{l.cone}). For a small positive number $\epsilon$ to be chosen, that determines
  the precision with which we approximate the cone transforms $T_{C_s}$,
  obtain operators $T_s$, polynomials in $\Theta \cup \bar{\Theta} \cup \{ 1
  \}$. 
  
  We are going to see
  that, indeed, the estimate
  \[ \|[T_1, ... [T_t, M_b] ...] \bar{\gamma} \|_2 \gtrsim 1 \]
  holds. The commutator consists of terms of the form $T \beta T' 
  \bar{\gamma}$ where $T, T'$ are combinations of $T_s$ and the identity. In
  the case where $T'$ is not the identity, it follows from
  lemma \ref{l.approximation} that the symbol of $T'$ is at most $\epsilon$ on the Fourier support
  of $\bar{\gamma}$. Such components are small:
  \[ \|T \beta T'  \bar{\gamma} \|_2 \lesssim \| \beta T'  \bar{\gamma} \|_2
     \lesssim \| \beta \|_4  \|T'  \bar{\gamma} \|_4 \lesssim \epsilon^{1 / 3}
  \]
  To obtain the last inequality, we observe the following: first, recall that
  $\beta$ is normalized both in $\tmop{BMO}$ and $L^2$. By interpolation we
  control $L^4$ norms uniformly. Observe also that $T'$ is at most $\epsilon$
  on the Fourier support of $\bar{\gamma}$, which gives us $\| T' \bar{\gamma}
  \|_2 \leq \epsilon$. In addition, $T'$ has universal $L^8$ norms independent
  of $\epsilon$. It is here that we use good approximation of the symbol
  controlling all derivatives. It remains to interpolate to obtain the
  estimate above.
  
  Now we are left with term $T \beta \bar{\gamma} = T_1 \ldots .T_t \beta
  \bar{\gamma}$ which we estimate as follows. Remember that $\gamma =
  T_{\vec{D}} \beta$ and write
  \[ \beta = \gamma + ( H_{\vec{D}} - T_{\vec{D}}) \beta + ( I - H_{\vec{D}})
     \beta, \]
  thus obtaining three terms. We will see that only one of them is large.
  
  The functions $( I - H_{\vec{D}}) \beta$ and $\bar{\gamma}$ are supported on
  the same product of half spaces complementary to cones $D_s$. We know that
  the symbol $c_{D, C}$ vanishes and therefore the $T_s$ are at most
  $\epsilon$, so
  \[ \| T ( ( I - H_{\vec{D}}) \beta \cdot \bar{\gamma}) \|_2 \leqslant
     \epsilon \| ( I - H_{\vec{D}}) \beta \cdot \bar{\gamma} \|_2 \leqslant
     \epsilon \| ( I - H_{\vec{D}}) \beta \|_4 \| \bar{\gamma} \|_4 . \]
  Recall the compositions of half plane projection operators have uniform
  $L^p$ bounds and that $L^4$ norms of both $\beta$ and $\gamma$ are
  controlled.
  
  For the part $T ( ( H_{\vec{D}} - T_{\vec{D}}) \beta \cdot \bar{\gamma})$
  we rely on the estimate from lemma \ref{l.cone} of the $L^4
  \tmop{norm}$
  \[  \| T ( ( H_{\vec{D}} - T_{\vec{D}}) \beta \cdot \bar{\gamma}) \|_2
     \lesssim \| ( H_{\vec{D}} - T_{\vec{D}}) \beta \cdot \bar{\gamma} \|_2
     \leqslant \kappa \| \bar{\gamma} \|_4 \lesssim \kappa . \]
  For the term $T ( \gamma \bar{\gamma})$we consider
  \[ | \| T \gamma \bar{\gamma} \|_2 - \| H_{\vec{C}} \gamma \bar{\gamma} \|_2
     | \leqslant \| ( T - H_{\vec{C}}) \gamma \bar{\gamma} \|_2 \leqslant \| (
     T - T_{\vec{C}}) \gamma \bar{\gamma} \|_2 + \| ( T_{\vec{C}} -
     H_{\vec{C}}) \gamma \bar{\gamma} \|_2 \lesssim \epsilon + \kappa \]
  Since $\gamma \bar{\gamma}$ is real with symmetric Fourier transform, we
  have $\| H_{\vec{C}} \gamma \bar{\gamma} \|_2 \gtrsim \| \gamma \bar{\gamma}
  \|_2 = \| \gamma \|^2_4$. Furthermore
  \[ \| \gamma \|_4^2 \gtrsim \left\| \left( \sum_{\varepsilon} \sum_{R \in
     \mathcal{U}} \frac{| \langle \gamma, w_R \rangle |^2}{| R |}
     \tmmathbf{1}_R  \right)^{1 / 2} \right\|_4^2 \gtrsim \left\| \left(
     \sum_{\varepsilon} \sum_{R \in \mathcal{U}} \frac{| \langle \gamma, w_R
     \rangle |^2}{| R |} \tmmathbf{1}_R  \right)^{1 / 2}
     \right\|_{2^{}}^2 \gtrsim 1 \]

  The first inequality uses a Littlewood Paley inequality and to see the
  second inequality, note that the rectangles in $\mathcal{U}$ are contained
  in a set of measure bounded by 1. We have therefore proved that $\| T (
  \gamma \bar{\gamma}) \|_2 \gtrsim 1.$

  We wish to prove that commutators that arise with our CZOs
  themselves are large, not just specific polynomials in those operators. To
  do so, observe the following elementary fact. Let $T, T'$ be CZO's. Then
  \[ [TT', M_b,] = T [T', M_b] + [T, M_b] T' . \]
  If the symbols of $T_s$ and $T_s'$ are polynomials in the $\theta_s$, it
  follows that for some choice of operators associated to $\theta_s$,
  \[ \| [ T_{1, k_1} [ \ldots . [ T_{t, k_t}, \beta]]] \bar{\gamma}' \|_2
     \gtrsim 1 \]
  where $\bar{\gamma}'$ is of the form $T \bar{\gamma}$ and where $T$ is a
  composition of operators $T_{s, l_s}$. Notice here that it is essential that
  we only approximate a finite set of cone operators so that we control
  degrees and coefficients of the arising polynomials.  This point is imperative, since we do not control
  degree or coefficients with Nachbin's approximation.  
  
  Recall that $\beta = P_{\mathcal{U}} b$ and that all dyadic rectangles are
  split into three groups $\mathcal{U} \dot{\cup \mathcal{V} \dot{\cup
  \mathcal{W}}}$. In order to see that the norm of the commutator satisfies
  $\| [ T_{1, k_1} [ \ldots . [ T_{t, k_t}, b]]] \|_{2 \rightarrow 2} \gtrsim
  1$, we use test function $\bar{\gamma}'$ and split the estimate according to
  partial sums of the symbol $b$ of only those rectangles belonging to classes
  $\mathcal{U}, \mathcal{V}, \mathcal{W}$ respectively. We have already seen
  that
  \[ \| [ T_{1, k_1} [ \ldots . [ T_{t, k_t}, P_{\mathcal{U}} b]]]
     \bar{\gamma}' \|_2 \gtrsim 1. \]
  It remains to see that the remaining parts are small. We are going to see
  that
  \[ \| [ T_{1, k_1} [ \ldots . [ T_{t, k_t}, P_{\mathcal{V}} b]]]
     \bar{\gamma}' \|_2 \lesssim \delta_J^{1 / 4}, \]
  the part of the estimate responsive to Journee's lemma and also that
  \[ \| [ T_{1, k_1} [ \ldots . [ T_{t, k_t}, P_{\mathcal{W}} b]]]
     \bar{\gamma}' \|_2 \lesssim \delta_{- 1} . \]
  For these two estimates, we can follow directly the arguments in LPPW.
  
  The first estimate illustrates the use of Journe's lemma in this context.
  We do not need to use any cancellation of the commutator:
  \[ \| [ T_{1, k_1} [ \ldots . [ T_{t, k_t}, P_{\mathcal{V}} b]]]
     \bar{\gamma}' \|_2 \lesssim \| P_{\mathcal{V}} b \|_4 \| \gamma \|_{_4}
  \]
  where the implied constant depends upon $L^2$ and $L^4$ operator norms of
  the $T_{s, k_s}$. The $L^4$ norm of $\gamma$ is uniformly controlled and by
  construction we have $\| P_{\mathcal{V}} b \|_{\tmop{BMO}} \leqslant 1$.
  Last, Journe's lemma provides us with the estimate $\| P_{\mathcal{V}} b
  \|_2^2 \leqslant \delta_J$. Interpolation then gives $\| P_{\mathcal{V}} b
  \|_4 \lesssim \delta_J^{1 / 4}$.
  
  The last estimate requires a careful analysis, but does not use the
  specifics of our operators, except the control on a large number of
  derivatives of the kernel. We therefore appeal to the version in LPPW, where
  the estimate was stated for Riesz transforms but in fact carried out for
  more general CZOs with control on a large number of derivatives, such as the
  ones we have here.

 \end{proof}

  \section{Concluding Remarks}
  
  \begin{remark}
    Our theorem is a generalization of the Riesz transform case, but it falls
    short of recovering the full Uchiyama-Li criterion in several parameters.
    Li's criterion only requires point separation of all pairs $\xi$ and $-
    \xi$ on the sphere. This criterion is quite natural as it makes sure there
    is an operator in the family that has a singularity in a given direction,
    for all directions. Due to the method of proof, we felt the need to
    require point separation for all pairs of points as well as a derivative
    condition. The strategy to obtain lower bounds in this multi parameter
    setting remains analytic in nature - while we are not able to use Fourier
    projections directly as in one dimension, we build operators that are
    close enough to still pretend we are in the one dimensional setting.
    Families that have Li's criterion are not enough to approximate the
    operators we need in the norm of uniform convergence in $\mathcal{C} (
    \mathcal{S}^{d_{} - 1})$ much less in $\mathcal{C}^n ( \mathcal{S}^{d_{} -
    1})$. We require the latter because we need excellent convergence of
    multiplier symbols on the Fourier transform side in order to draw
    meaningful conclusions. It is interesting to remark that, in cases like
    ours, one easily proves a version of Stone Weierstrass theorem that can
    handle defects in the sense that it is clear which algebra is generated by
    a family of functions with defects, such as a lack of point separation for
    a given pair of $\xi$ and $\zeta$ in$\mathcal{S}^{d_{} - 1}$. One uses
    factor spaces to see that the generated algebra will have the exact same
    set of defects: the algebra generated by a family that lacks point
    separation for a set of pairs $( \xi, \zeta)$ will be the subalgebra with
    that same property. The situation is not so simple if one needs uniform
    approximation in $\mathcal{C}^n ( \mathcal{S}^{d_{} - 1})$. Due to the
    necessary conditions on the tangential derivatives, the situation becomes
    very complex when the family has defects, such as a lack of point
    separation in just one point or the lack of non-zero tangential
    derivatives. The corresponding subalgebras are unknown since the 1950s.
  \end{remark}

\end{document}